\documentclass[12pt,reqno]{amsart}
\usepackage{amsfonts}
\usepackage{amssymb,color}
\usepackage{amssymb}

\usepackage[colorlinks=true]{hyperref}
\hypersetup{urlcolor=blue, citecolor=blue}

\usepackage[margin=2.2cm, a4paper]{geometry}

\newtheorem{theo}{Theorem}[section]
\newtheorem{lemm}[theo]{Lemma}
\newtheorem{defi}[theo]{Definition}

\numberwithin{equation}{section}

\newcommand{\bal}{\begin{align}}
\newcommand{\bbal}{\begin{align*}}
\newcommand{\beq}{\begin{equation}}
\newcommand{\eeq}{\end{equation}}
\newcommand{\bca}{\begin{cases}}
\newcommand{\eca}{\end{cases}}

\newcommand{\pa}{\partial}

\newcommand{\na}{\nabla}
\newcommand{\De}{\Delta}
\newcommand{\de}{\delta}

\newcommand{\cd}{\cdot}
\newcommand{\ep}{\varepsilon}
\newcommand{\dd}{\ \mathrm{d}}
\newcommand{\B}{\dot{B}}
\newcommand{\LL}{\tilde{L}}
\newcommand{\R}{\mathbb{R}}

\newcommand{\D}{\mathrm{div} \ }
\newcommand{\uu}{\mathbf{u}}

\newcommand{\U}{\mathbf{U}}

\newcommand{\XX}{\mathrm{X}}
\newcommand{\YY}{\mathrm{Y}}
\newcommand{\DDe}{\dot{\Delta}}

\newcommand{\Z}{\mathbb{Z}}

\begin{document}


\title[Ill-posedness for the viscous shallow water equations]{Ill-posedness for the 2D viscous shallow water equations in the critical Besov spaces}

\author[J. Li]{Jinlu Li}
\address{School of Mathematics and Computer Sciences, Gannan Normal University, Ganzhou 341000, China}
\email{lijinlu@gnnu.edu.cn}

\author[P. Hong]{Pingzhou Hong}
\address{School of Mathematics and Computer Sciences, Gannan Normal University, Ganzhou 341000, China}
\email{pzhong-66@163.com}

\author[W. Zhu]{Weipeng Zhu}
\address{School of Mathematics and Information Science, Guangzhou University, Guangzhou 510006, China}
\email{mathzwp2010@163.com}

\begin{abstract}
In this paper, we prove that the 2D viscous shallow water equations is ill-posed in the critical Besov spaces $\B^{\frac2p-1}_{p,1}(\R^2)$ with $p>4$. Our proof mainly depends on the method introduced by Bourgain-Pavlovi\'c in \cite{B-P} and Chen-Miao-Zhang in \cite{C-M-Z4}.
\end{abstract}

\maketitle

\section{Introduction and main result}

In the paper, we consider the following 2D viscous shallow water equations:
\bal\label{1.1}\bca
\pa_th+\mathrm{div}(h\uu)=0,\\
h(\pa_t\uu+\uu\cdot\nabla\uu)-\nu\nabla\cdot(h\nabla \uu)+h\nabla h=0, \\
h(0,x)=h_0,\quad \uu(0,x)=\uu_0,
\eca\end{align}
where $h(t,x)$ represents the height of the fluid level above the solid bed, $\uu(t,x)=(\uu^1(t,x),\uu^2(t,x))$ is the horizontal velocity field and $\nu>0$ is the viscous coefficient.

The viscous shallow water equations have been widely studied by mathematicians, cf., \cite{D, D.B} and references therein. In \cite{bui}, by using Lagrangian coordinates and H\"{o}lder space estimates, Bui obtained the local existence and uniqueness of classical solutions
to the Cauchy-Dirichlet problem for \eqref{1.1} with initial data in $C^{2+\alpha}$. By using the energy method of Matsumura and Nishida \cite{M-N}, Kloeden \cite{K} and Sundbye \cite{S1} showed the global existence and uniqueness of classical solutions to the Cauchy-Dirichlet problem for \eqref{1.1}. Subsequently, the existence and uniqueness of classical solutions to the Cauchy problem for \eqref{1.1} was also proved by Sundbye \cite{S2}. By applying the Littlewood-Paley decomposition
theory for Sobolev spaces to obtain a losing energy estimate in $H^{s+2}$ for any $s>0$, Wang and Xu \cite{W} showed that the solution of \eqref{1.1} exists locally and uniquely for all initial data $u_0$ and exists globally for small initial data $u_0$ if $h_0-\bar{h}_0$ is small enough. Lately, Liu and Yin \cite{L-Y,L-Y2,L-Y1} improved the result of \cite{W} in the Sobolev spaces with low regularity and inhomogeneous Besov spaces.

In \cite{C-M-Z3}, Chen, Miao and Zhang developed a new method which relies on the smoothing
properties of the heat equations and introduced some kind of weighted Besov norms to study
the well-posedness of \eqref{1.1} for the initial data with the minimal regularity and prove the
local well-posedness under the more natural assumption that the initial height is bounded
away from zero. They obtained the local well-posedness of the system \eqref{1.1} for general initial
data in critical Besov spaces with $L^p$ type. Now, let us recall some important progress about
the global existence results for small data. Motivated by the ideas of Danchin \cite{D1,D2}, Chen-Miao-Zhang \cite{C-M-Z} proved the global well-posedness of the system \eqref{1.1} for small initial data
in critical Besov spaces with $L^2$ type. Chen-Miao-Zhang \cite{C-M-Z2} and Charve-Danchin \cite{C-D} obtained
the global well-posedness of the system \eqref{1.1} in the hybrid Besov spaces, in which the part of
high frequency of the initial data lies in the critical Besov spaces with $L^p$ type integrability. Subsequently, Haspot \cite{H1} improved the results of \cite{C-D,C-M-Z2} by a smart use of the viscous effective flux. For more results of the solutions to \eqref{1.1}, we refer the reader to see \cite{Dan17,Fan18,H2}.

For the sake of convenience, we take $\bar{h}_0=1$ and $\nu=1$. Substituting $h$ by $1+h$ in \eqref{1.1}, we have
\bal\label{1.2}\bca
\pa_th+\D \uu+\uu\cd \na h=h\D\uu,\\
\pa_t\uu+\uu\cd \na \uu-\De \uu+\na h=\na(\ln(1+h))\cd \na \uu\\
h(0,x)=h_0,\quad \uu(0,x)=\uu_0.
\eca\end{align}

By \cite{C-M-Z3}, the system \eqref{1.2} is locally well-posed in the critical Besov space with $1\leq p<4$. However, the problem whether the system \eqref{1.2} is well-posedness in the critical Besov space with $p>4$ is unsolved.  Recently,  Chen-Miao-Zhang \cite{C-M-Z4} proved the ill-posedness of the 3D compressible Navier-Stokes equations in critical Besov spaces with $p>6$. Motivated by Bourgain-Pavlovi\'c \cite{B-P} and Chen-Miao-Zhang \cite{C-M-Z4}, we prove that the system \eqref{1.2} is ill-posed in the critical Besov spaces with $p>4$. Our main ill-posedness result reads as follows.
\begin{theo}\label{th1}
Let $p>4$. For any $\de>0$, there exists initial data satisfying
\bbal
||h_0||_{\B^{\frac2p}_{p,1}}+||\uu_0||_{\B^{\frac2p-1}_{p,1}}\leq \de,
\end{align*}
such that a solution $(h,\uu)$ for the system \eqref{1.2} satisfies
\bbal
||\uu(t)||_{\B^{\frac2p-1}_{p,1}}\geq \frac1\de, \quad \mathrm{for} \ \mathrm{some} \quad 0<t<\de.
\end{align*}
\end{theo}

Our paper is organized as follows. In Section 2, we give some preliminaries which will be used in the sequel. In Section 3, we will give the proof of Thorem \ref{th1}.\\

\noindent\textbf{Notation.} In the following, since all spaces of functions are over $\mathbb{R}^2$, for simplicity, we drop $\mathbb{R}^2$ in our notations of function spaces if there is no ambiguity. Let $C\geq 1$ and $c\leq 1$ denote constants which can be different at different places. We use $A\lesssim B$ to denote $A\leq CB$.

\section{Littlewood-Paley analysis}

In this section, we first recall some tools from the Littlewood-Paley theory, the definition of homogeneous Besov spaces and some useful properties. Then, we state some applications in the linear transport equation and the heat equation.

First, let us introduce the Littlewood-Paley decomposition. Choose a radial function $\varphi\in \mathcal{S}(\mathbb{R}^2)$ supported in $\tilde{\mathcal{C}}=\{\xi\in\mathbb{R}^2,\frac34\leq \xi\leq \frac83\}$ such that
\begin{align*}
\sum_{j\in \mathbb{Z}}\varphi(2^{-j}\xi)=1 \quad \mathrm{for} \ \mathrm{all} \ \xi\neq0.
\end{align*}
The frequency localization operator $\dot{\Delta}_j$ and $\dot{S}_j$ are defined by
\begin{align*}
\dot{\Delta}_jf=\varphi(2^{-j}D)f=\mathcal{F}^{-1}(\varphi(2^{-j}\cdot)\mathcal{F}f), \quad \dot{S}_jf=\sum_{k\leq j-1}\dot{\Delta}_kf \quad \mathrm{for} \quad j\in\mathbb{Z}.
\end{align*}
With a suitable choice of $\varphi$, one can easily verify that
\begin{align*}
\dot{\Delta}_j\dot{\Delta}_kf=0 \quad \mathrm{if} \quad |j-k|\geq2, \quad \dot{\Delta}_j(\dot{S}_{k-1}f\dot{\Delta}_kf)=0 \quad  \mathrm{if} \quad  |j-k|\geq5.
\end{align*}
Next we recall Bony's decomposition from \cite{B.C.D}:
\begin{align*}
uv=\dot{T}_uv+\dot{T}_vu+\dot{R}(u,v),
\end{align*}
with
\begin{align*}
\dot{T}_uv=\sum_{j\in\mathbb{Z}}\dot{S}_{j-1}u\dot{\Delta}_jv, \quad \quad \dot{R}(u,v)=\sum_{j\in\mathbb{Z}}\dot{\Delta}_ju\widetilde{\Delta}_jv, \quad \quad \widetilde{\Delta}_jv=\sum_{|j'-j|\leq 1}\dot{\Delta}_{j'}v.
\end{align*}

The following Bernstein lemma will be stated as follows:
\begin{lemm} (\cite{B.C.D})\label{le1}
Let $1\leq p\leq q\leq \infty$  and $\mathcal{B}$ be a ball and $\mathcal{C}$ a ring of $\mathbb{R}^2$. Assume that $f\in L^p$, then for any $\alpha\in\mathbb{N}^2$, there exists a constant $C$ independent of $f$, $j$ such that
\begin{align*}
&&\mathrm{Supp} \,\hat{f}\subset\lambda \mathcal{B}\Rightarrow\sup_{|\alpha|=k}\|\partial^{\alpha}f\|_{L^q}\le C^{k+1}\lambda^{k+2(\frac1p-\frac1q)}\|f\|_{L^p},\\
&&\mathrm{Supp} \,\hat{f}\subset\lambda \mathcal{C}\Rightarrow C^{-k-1}\lambda^k\|f\|_{L^p}\le\sup_{|\alpha|=k}\|\partial^{\alpha}f\|_{L^p}
\le C^{k+1}\lambda^{k}\|f\|_{L^p}.
\end{align*}
\end{lemm}

We denote by $\mathcal{S}'_h(\R^2)$ the set of the tempered distribution $f$ satisfying $\lim\limits_{\lambda\rightarrow+\infty}||\chi(\lambda D)f||_{L^\infty}=0$ for some $\chi\in \mathcal{D}(\R^2)$ and $\chi(0)=0$. Let us introduce the homogeneous Besov spaces and Chemin-Lerner type Bseov spaces.

\begin{defi}
Let $s\in \mathbb{R}$, $1\leq p,r\leq \infty$. The homogeneous Besov space $\B^s_{p,r}$ is defined by
\begin{align*}
\B^s_{p,r}=\{f\in \mathcal{S}'_h(\mathbb{R}^2):||f||_{\B^s_{p,r}}<+\infty\},
\end{align*}
where
\begin{align*}
||f||_{\B^s_{p,r}}\triangleq \Big|\Big|(2^{ks}||\dot{\Delta}_k f||_{L^p(\R^2)})_{k}\Big|\Big|_{\ell^r}.
\end{align*}
\end{defi}

\begin{defi}
Let $s\in\mathbb{R}$, $1\leq p,q,r\leq\infty$ and $T\in(0,\infty]$. The Chemin-Lerner type Besov space $\tilde{L}^q_T(\B^s_{p,r})$ is defined as the set of all the distributions $f$ satisfying
\begin{align*}
||f||_{\tilde{L}^q_T(\B^s_{p,r})}\triangleq \Big|\Big|\big(2^{ks}||\dot{\Delta}_kf||_{L^q_T(L^p(\R^2))}\big)_k \Big|\Big|_{\ell^r}<+\infty.
\end{align*}
\end{defi}
By Minkowski's inequality, it is easy to find that
\begin{align*}
||f||_{\tilde{L}^q_T(\B^s_{p,r})}\leq ||f||_{L^q_T(\B^s_{p,r})} \quad \mathrm{if} \quad q\leq r, \quad \quad ||f||_{\tilde{L}^q_T(\B^s_{p,r})}\geq ||f||_{L^q_T(\B^s_{p,r})} \quad \mathrm{if} \quad q\geq r.
\end{align*}

Let us present the priori estimates of the linear transport equation
\begin{align}\label{3.1}
\partial_tf+v\cdot\nabla f=g,\quad f(0,x)=f_0,
\end{align}
and the heat equation
\begin{align}\label{3.2}
\partial_tu-\Delta u=G, \quad u(0,x)=u_0,
\end{align}
in homogenous Besov spaces. The following estimates will be frequently used in the sequel.
\begin{lemm}(\cite{D3})\label{le3.1}
Let $s\in(-2\min\{\frac1p,1-\frac{1}{p}\}-1,1+\frac 2p]$ and $1\leq p\leq \infty$. Let $v$ be a vector field such that $\nabla v\in L^1_T(\B^{\frac2p}_{p,1})$. Assume that $f_0\in\B^s_{p,1}$, $g\in L^1_T(\B^s_{p,1})$ and $f$ is the solution of \eqref{3.1}. Then there holds for $t\in[0,T]$,
\begin{align*}
||f||_{\tilde{L}^\infty_t(\B^s_{p,1})}\leq e^{CV(t)}(||f_0||_{\B^s_{p,1}}+\int^t_0e^{-CV(\tau)}||g(\tau)||_{\B^s_{p,1}}\mathrm{d}\tau),
\end{align*}
or
\begin{align*}
||f||_{\tilde{L}^\infty_t(\B^s_{p,1})}\leq e^{CV(t)}(||f_0||_{\B^s_{p,1}}+||g||_{\tilde{L}^1_t(\B^s_{p,1})}),
\end{align*}
where $V(t)=\int^t_0||\nabla v||_{\B^{\frac 2p}_{p,1}}\mathrm{d} \tau$.
\end{lemm}

\begin{lemm}(\cite{D3})\label{le3.2}
Let $s\in \mathbb{R}$ and $1\leq q,q_1,p,r\leq \infty$ with $q_1\leq q$. Assume that $u_0\in \B^s_{p,r}$ and $G\in \tilde{L}^{q_1}_T(\B^{s-2+\frac{2}{q_1}}_{p,r})$. Then \eqref{3.2} has a unique solution $u\in \tilde{L}^{q}_T(\B^{s+\frac2q}_{p,r})$ satisfying
\begin{align*}
||u||_{\tilde{L}^{q}_T(\B^{s+\frac2q}_{p,r})}\leq C(||u_0||_{\B^s_{p,r}}+||G||_{\tilde{L}^{q_1}_T(\B^{s-2+\frac{2}{q_1}}_{p,r})}).
\end{align*}
\end{lemm}

Finally, we need the following estimates for the product estimates in the next section.

\begin{lemm}{\cite{B.C.D}}\label{le11}
Let $T,s>0$ and $1\leq ,\rho\leq \infty$. Then it holds that
\bbal
||fg||_{\LL^\rho_T(\B^s_{p,1})}\leq C\big(||g||_{L^\infty_T(L^\infty)}||f||_{\LL^\rho_T(\B^s_{p,1})}+||f||_{L^\infty_T(L^\infty)}||g||_{\LL^\rho_T(\B^s_{p,1})}\big).
\end{align*}
\end{lemm}

\begin{lemm}{\cite{B.C.D,C-M-Z4}}\label{le12}
Let $T,s>0$ and $1\leq ,\rho\leq \infty$. Assume that $F\in W_{loc}^{[\sigma]+3}(\R)$ with $F(0)=0$. Then for any $f\in L^\infty\cap \B^s_{p,1}$, we have
\bbal
||F(f)||_{\LL^\rho_T(\B^s_{p,1})}\leq C\big(1+||f||_{L^\infty_T(L^\infty)}\big)^{[\sigma]+2}||f||_{\LL^\rho_T(\B^s_{p,1})}.
\end{align*}
\end{lemm}

\begin{lemm}\label{le3.3}
Let $1\leq \rho,\rho_1,\rho_2\leq \infty$ with $\frac1\rho=\frac1\rho_1+\frac1\rho_2$ and $2\leq p\leq 4\leq q < \infty$ with $\frac2p+\frac 2q>1$. Then, we have
\bbal
||fg||_{\LL^\rho_T(\B^{\frac 2p-1}_{p,1})}\lesssim ||f||_{\LL^{\rho_1}_T(\B^{\frac 2p-1}_{p,1})}||g||_{\LL^{\rho_2}_T(\B^{\frac 2q}_{q,1})}.
\end{align*}
\end{lemm}
\begin{proof}
Let $\frac1q+\frac{1}{p}=\frac1r$. According to H\"{o}lder's inequallity and Lemma \ref{le1}, we deduce that
\bbal
||\dot{T}_fg||_{\LL^\rho_T(\B^{\frac 2p-1}_{p,1})}&\leq \sum_{k\in \Z} 2^{k(\frac2p-1)}||\dot{S}_{k-1}f||_{L^{\rho_1}_T(L^{\frac{pq}{q-p}})}||\dot{\De}_kg||_{L^{\rho_2}_T(L^q)}
\\&\leq \sum_{k\in \Z} \sum_{k'\leq k-1}2^{k'(\frac2p-1)}||\dot{\De}_{k'}f||_{L^{\rho_1}_T(L^{p})}\cdot2^{\frac{2k}{q}}||\dot{\De}_kg||_{L^{\rho_2}_T(L^q)}\cdot 2^{(k-k')(\frac2p-\frac2q-1)}
\\&\leq C||f||_{\LL^{\rho_1}_T(\B^{\frac 2p-1}_{p,1})}||g||_{\LL^{\rho_2}_T(\B^{\frac 2q}_{q,1})},
\end{align*}
\bbal
||\dot{T}_gf||_{\LL^\rho_T(\B^{\frac 2p-1}_{p,1})}&\leq \sum_{k} 2^{k(\frac2p-1)}||\dot{S}_{k-1}g||_{L^{\rho_1}_T(L^{\infty})}||\DDe_kg||_{L^{\rho_2}_T(L^p)}
\\&\leq \sum_{k}2^{k'(\frac2p-1)}||\DDe_{k}f||_{L^{\rho_1}_T(L^{p})}\cdot \sum_{k'}||\DDe_{k'}g||_{L^{\rho_2}_T(L^\infty)}
\\&\leq C||f||_{\LL^{\rho_1}_T(\B^{\frac 2p-1}_{p,1})}||g||_{\LL^{\rho_2}_T(\B^{\frac 2q}_{q,1})},
\end{align*}
\bbal
||\dot{R}(f,g)||_{\LL^\rho_T(\B^{\frac 2p-1}_{p,1})}&\leq \sum_k\sum_{k'\geq k-3}2^{k(\frac2p+\frac2q-1)}||\DDe_{k'}f||_{L^{\rho_1}_T(L^{p})}||\DDe_{k'}g||_{L^{\rho_2}_T(L^q)}
\\&\leq \sum_k\sum_{k'\geq k-3}2^{k'(\frac2p-1)}||\DDe_{k'}f||_{L^{\rho_1}_T(L^{p})}\cdot 2^{\frac{2k'}{q}}||\DDe_{k'}g||_{L^{\rho_2}_T(L^q)}\cdot 2^{(k-k')(\frac2p+\frac2q-1)}
\\&\leq C||f||_{\LL^{\rho_1}_T(\B^{\frac 2p-1}_{p,1})}||g||_{\LL^{\rho_2}_T(\B^{\frac 2q}_{q,1})}.
\end{align*}
This completes the proof of this lemma.
\end{proof}

\section{Proof of the main theorem}

In this section, we will give the details for the proof of the main theorem.

Let $p>4$. For simplicity, we define the following index, that is
\bbal
\frac{2}{p*}+\frac 2p=1, \quad \frac {4}{q}=\frac{2}{p*}+\frac12, \quad \frac{2}{q*}+\frac 2q=1, \quad \frac{4}{r}=\frac{2}{q*}+\frac12.
\end{align*}
Then, we have
\bal\label{l-z}\begin{split}
&2\leq p*<q<4<r<q*<p, \quad \frac 2r-\frac 2p=\frac{1}{2}(\frac 2q-\frac 2p), \quad \frac34(1-\frac4p)=\frac2q-\frac2p,
\\& \frac 2q+\frac 2r>1.
\end{split}\end{align}
Define a smooth function $\phi$ with values in $[0,1]$ which satisfies
\bbal
\phi(\xi)=
\bca
1, \quad \mathrm{if} \ |\xi|\leq \frac14,\\
0, \quad \mathrm{if} \ |\xi|\geq \frac12.
\eca
\end{align*}
Motivated by  \cite{B-P,C-M-Z4}, we introduce the initial data $(h_0,\uu_0)$ as
\bbal
h_0=0,\quad \hat{\uu}_0=2^{n(1-\frac 2p-\ep)}\big(\phi(\xi-2^n\vec{e})+\phi(\xi+2^n\vec{e}),i\phi(\xi-2^n\vec{e})-i\phi(\xi+2^n\vec{e})\big),
\end{align*}
where $\vec{e}=(1,1)$ and $\ep$ satisfying $\ep=\frac{19}{30}(\frac14-\frac 1p))>0$.  Here, we also have
\bbal
6\ep<1-\frac 4p, \quad 5\ep>\frac2q-\frac2p=\frac34(1-\frac 4p).
\end{align*}
Notice that
\bbal
\uu_0&=2^{n(1-\frac 2p-\ep)}\big(e^{i2^nx\cd \vec{e}}+e^{-i2^nx\cd \vec{e}},ie^{i2^nx\cd \vec{e}}-ie^{-i2^nx\cd \vec{e}}\big)\check{\phi}(x)
\\&=2^{n(1-\frac 2p-\ep)+1}\big(\cos(2^nx\cd \vec{e}),-\sin(2^nx\cd \vec{e})\big)\check{\phi}(x).
\end{align*}
Obviously, the initial data $\uu_0$ is a real-valued field. Now we decompose the solution $\uu(t)$ into $\uu(t)=\U_0(t)+\U_1(t)+\U_2(t)$ satisfying
\bbal
\U_0(t)=e^{t\De}\uu_0, \quad \quad \U_1=-\int^t_0e^{(t-\tau)\De}(\U_0\cd \na \U_0)\dd \tau,
\end{align*}
and
\bbal
\pa_t\U_2-\De \U_2&=\mathbf{F}_1+\na \ln(1+h)\cd\na(\U_0+\U_1+\U_2)-\na h,
\end{align*}
where
\bbal
\mathbf{F}_1&=-(\U_0\cd \na \U_1+\U_0\cd \na \U_2+\U_1\cd \na \U_0+\U_1\cd \na \U_1
\\&\quad +\U_1\cd \na \U_2+\U_2\cd \na \U_0+\U_2\cd \na \U_1+\U_2\cd \na \U_2).
\end{align*}
It is easy to show that for any $s\geq 0$ and $q_0\geq 2$,
\bal\label{zhu-0}
||\U_0||_{\LL^\infty_T(\B^{\frac {2}{q_0}-1+s}_{q_0,1})\cap L^1_T(\B^{\frac {2}{q_0}+1+s}_{q_0,1})}\leq C||\uu_0||_{\B^{\frac {2}{q_0}-1+s}_{q_0,1}}\leq C2^{n(\frac {2}{q_0}-\frac 2p+s-\ep)}.
\end{align}
Now, we establish the lower bound estimate of $||\U_1||_{\B^{\frac 2p-1}_{p,1}}$. Since $\B^{\frac 2p-1}_{p,1}\hookrightarrow \B^{-1}_{\infty,\infty}$, we have $||\U_1||_{\B^{\frac 2p-1}_{p,1}}\geq c|\int_{\R^2}\varphi(16\xi)\hat{\U}_1(\xi)\dd \xi|$ for some $c>0$ independent of $n$. Note that
\bbal
&\big(\U_0\cd \na \U_0\big)^1=\U^1_0\pa_1\U^1_0+\U^2_0\pa_2\U^1_0, \quad \big(\U_0\cd \na \U_0\big)^2=\U^1_0\pa_1\U^2_0+\U^2_0\pa_2\U^2_0.
\end{align*}
It follows from \eqref{zhu-0} that
\bbal
||\int^t_0e^{(t-\tau)\De}(\U^1_0\pa_1 \U^1_0)\dd \tau||_{\B^{\frac 2p-1}_{p,1}}+||\int^t_0e^{(t-\tau)\De}(\U^2_0\pa_2 \U^2_0)\dd \tau||_{\B^{\frac 2p-1}_{p,1}}\leq C2^{-2\ep n}.
\end{align*}
Therefore, we can show that
\bal\begin{split}\label{hong-1}
||\U_1||_{\B^{\frac 2p-1}_{p,1}}&\geq c\Big(|\int_{\R^2}\int^t_0e^{-(t-\tau)|\xi|^2}\varphi(16\xi)\mathcal{F}\big(e^{\tau\De}\uu^2_0\pa_2e^{\tau\De}\uu^1_0\big)(\xi) \dd \tau\dd \xi|
\\& \quad +|\int_{\R^2}\int^t_0e^{-(t-\tau)|\xi|^2}\varphi(16\xi)\mathcal{F}\big(e^{\tau\De}\uu^1_0\pa_1e^{\tau\De}\uu^2_0\big)(\xi) \dd \tau\dd \xi|\Big)-C2^{-2\ep n}.
\end{split}\end{align}
Since $\mathrm{supp} \ \phi(\cdot-a)*\phi(\cdot-b)\subset B(a+b,1)$, we see that
\bbal
&\mathrm{supp} \ \phi(\xi-2^n\vec{e})*\phi(\xi-2^n\vec{e})\subset B(2^{n+1}\vec{e},1),
\\& \mathrm{supp} \ \phi(\xi+2^n\vec{e})*\phi(\xi+2^n\vec{e})\subset B(-2^{n+1}\vec{e},1).
\end{align*}
Then, we can rewrite
\bbal
&\quad \ 2^{2n(\frac 2p-1+\ep)}\varphi(16\xi)\mathcal{F}\big(e^{\tau\De}\uu^2_0\pa_2e^{\tau\De}\uu^1_0\big)(\xi)\\&=
-\varphi(16\xi)\int_{\R^2}\eta_2e^{-\tau(|\xi-\eta|^2+|\eta|^2)}\phi(\xi-\eta-2^n\vec{e})\phi(\eta+2^n\vec{e}) \dd \eta
\\& \quad +\varphi(16\xi)\int_{\R^2}\eta_2e^{-\tau(|\xi-\eta|^2+|\eta|^2)}\phi(\xi-\eta+2^n\vec{e})\phi(\eta-2^n\vec{e}) \dd \eta,
\end{align*}
and
\bbal
&\quad \ 2^{2n(\frac 2p-1+\ep)}\varphi(16\xi)\mathcal{F}\big(e^{\tau\De}\uu^1_0\pa_1e^{\tau\De}\uu^2_0\big)(\xi)\\&=
\varphi(16\xi)\int_{\R^2}\eta_1e^{-\tau(|\xi-\eta|^2+|\eta|^2)}\phi(\xi-\eta-2^n\vec{e})\phi(\eta+2^n\vec{e}) \dd \eta
\\& \quad-\varphi(16\xi)\int_{\R^2}\eta_1e^{-\tau(|\xi-\eta|^2+|\eta|^2)}\phi(\xi-\eta+2^n\vec{e})\phi(\eta-2^n\vec{e}) \dd \eta.
\end{align*}
Note that
\bbal
\int^t_0e^{-(t-\tau)|\xi|^2}e^{-\tau(|\xi-\eta|^2+|\eta|^2)} \dd \tau=\frac{e^{-t|\xi|^2}-e^{-t(|\eta|^2+|\xi-\eta|^2)}}{|\eta|^2+|\xi-\eta|^2-|\xi|^2}.
\end{align*}
Hence, we can show that
\bbal
&\quad \ \int^t_0e^{-(t-\tau)|\xi|^2}\varphi(16\xi)\mathcal{F}\big(e^{\tau\De}\uu^2_0\pa_2e^{\tau\De}\uu^1_0\big)(\xi) \dd \tau\\&=-2^{2n(1-\frac 2p-\ep)}\varphi(16\xi)\int_{\R^2}\frac{e^{-t|\xi|^2}-e^{-t(|\xi-\eta|^2+|\eta|^2)}}
{|\xi-\eta|^2+|\eta|^2-|\xi|^2}\Big(\eta_2\phi(\xi-\eta-2^n\vec{e})\phi(\eta+2^n\vec{e})
\\& \quad -\eta_2\phi(\xi-\eta+2^n\vec{e})\phi(\eta-2^n\vec{e})\Big) \dd \eta
\\&=2^{2n(1-\frac 2p-\ep)}\varphi(16\xi)\int_{\R^2}\frac{e^{-t|\xi|^2}-e^{-t(|\xi-\eta|^2+2|\eta|^2)}}
{|\xi-\eta|^2+|\eta|^2-|\xi|^2}
(-2\eta_2+\xi_2)\phi(\xi-\eta-2^n\vec{e})\phi(\eta+2^n\vec{e}) \dd \eta,
\end{align*}
and
\bbal
&\quad \ \int^t_0e^{-(t-\tau)|\xi|^2}\varphi(16\xi)\mathcal{F}\big(e^{\tau\De}\uu^1_0\pa_1e^{\tau\De}\uu^2_0\big)(\xi) \dd \tau\\&=2^{2n(1-\frac 2p-\ep)}\varphi(16\xi)\int_{\R^2}\frac{e^{-t|\xi|^2}-e^{-t(|\xi-\eta|^2+|\eta|^2)}}
{|\eta-\xi|^2+|\eta|^2-|\xi|^2}\Big(\eta_1\phi(\xi-\eta-2^n\vec{e})\phi(\eta+2^n\vec{e})
\\& \quad -\eta_1\phi(\xi-\eta+2^n\vec{e})\phi(\eta-2^n\vec{e})\Big) \dd \eta
\\&=2^{2n(1-\frac 2p-\ep)}\varphi(16\xi)\int_{\R^2}\frac{e^{-t|\xi|^2}-e^{-t(|\xi-\eta|^2+2|\eta|^2)}}
{|\xi-\eta|^2+|\eta|^2-|\xi|^2} (2\eta_1-\xi_1)\phi(\xi-\eta-2^n\vec{e})\phi(\eta+2^n\vec{e}) \dd \eta.
\end{align*}
Making a change of variable, we obtain
\bbal
&\quad \ \int^t_0e^{-(t-\tau)|\xi|^2}\varphi(16\xi)\mathcal{F}\big(e^{\tau\De}\uu^2_0\pa_2e^{\tau\De}\uu^1_0\big)(\xi)\dd \tau\\&=2^{2n(1-\frac 2p-\ep)}\varphi(16\xi)\int_{\R^2}\frac{e^{-t|\xi|^2}-e^{-t(|\xi-\eta+2^n\vec{e}|^2+|\eta-2^n\vec{e}|^2)}}
{|\xi-\eta+2^n\vec{e}|^2+|\eta-2^n\vec{e}|^2-|\xi|^2}(-2\eta_2+2^{n+1}+\xi_2)\phi(\xi-\eta)\phi(\eta)\dd \eta,
\end{align*}
and
\bbal
&\quad \ \int^t_0e^{-(t-\tau)|\xi|^2}\varphi(16\xi)\mathcal{F}\big(e^{\tau\De}\uu^1_0\pa_1e^{\tau\De}\uu^2_0\big)(\xi) \dd \tau\\&=2^{2n(1-\frac 2p-\ep)}\varphi(16\xi)\int_{\R^2}\frac{e^{-t|\xi|^2}-e^{-t(|\xi-\eta+2^n\vec{e}|^2+|\eta-2^n\vec{e}|^2)}}
{|\xi-\eta+2^n\vec{e}|^2+|\eta-2^n\vec{e}|^2-|\xi|^2}(2\eta_1-2^{n+1}-\xi_1)\phi(\xi-\eta)\phi(\eta)\dd \eta.
\end{align*}
Using the Taylor's formula, we infer that for $t\leq 2^{-2n}$,
\bbal
\frac{e^{-t|\xi|^2}-e^{-t\big(|\xi-\eta+2^n\vec{e}|^2+|\eta-2^n\vec{e}|^2\big)}}
{|\xi-\eta+2^n\vec{e}|^2+|\eta-2^n\vec{e}|^2-|\xi|^2}=te^{-t|\xi|^2}\big(1+O(t2^{2n})\big),
\end{align*}
which along with $t=2^{-n(2+4\ep)}$ leads to
\bal\begin{split}\label{hong-2}
& \quad \ |\int_{\R^2}\int^t_0e^{-(t-\tau)|\xi|^2}\varphi(16\xi)\mathcal{F}\big(e^{\tau\De}\uu^2_0\pa_2e^{\tau\De}\uu^1_0\big)(\xi) \dd t\dd \xi|
\\& \geq ct2^n2^{2n(1-\frac 2p-\ep)}\geq c2^{n(1-\frac {4}{p}-6\ep)},
\end{split}\end{align}
and
\bal\begin{split}\label{hong-3}
& \quad \ |\int_{\R^2}\int^t_0e^{-(t-\tau)|\xi|^2}\varphi(16\xi)\mathcal{F}\big(e^{\tau\De}\uu^1_0\pa_1e^{\tau\De}\uu^2_0\big)(\xi)\ \dd t\dd \xi|
\\& \geq ct2^n2^{2n(1-\frac 2p-\ep)}\geq c2^{n(1-\frac {4}{p}-6\ep)},
\end{split}\end{align}
for some $c>0$ independent of $n$. Combining this results \eqref{hong-1}-\eqref{hong-3}, we have
\bal\label{hong-4}
||\U_1(t)||_{\B^{\frac 2p-1}_{p,1}}\geq c2^{n(1-\frac {4}{p}-6\ep)}-C2^{-2\ep n}\geq c2^{n(1-\frac {4}{p}-6\ep)}, \qquad t=2^{-n(2+4\ep)}, \ n\gg1,
\end{align}
for some $c>0$ independent of $n$. By Lemma \ref{le3.2}, we also notice that for $T\leq 2^{-2n-4\ep n}$,
\bal\begin{split}\label{zhu-1}
&||\U_0||_{\LL^2_T(\B^{\frac 2q}_{q,1})}\lesssim T^{\frac12}||\uu_0||_{\B^{\frac 2q}_{q,1}}\lesssim 2^{n(\frac 2q-\frac 2p-3\ep)},
\\&||\U_0||_{L^1_T(\B^{\frac 2q+1}_{q,1})}\lesssim T||\uu_0||_{\B^{\frac 2q+1}_{q,1}}\lesssim 2^{n(\frac 2q-\frac 2p-5\ep)},
\end{split}\end{align}
and
\bal\label{zhu-2}
||\U_0||_{\LL^2_T(\B^{\frac 2r}_{r,1})\cap L^1_T(\B^{\frac 2r+1}_{r,1})}\lesssim T^{\frac12}||\uu_0||_{\B^{\frac 2r}_{r,1}}\lesssim 2^{n(\frac 2r-\frac 2p-3\ep)}\lesssim 2^{\frac12n(\frac 2q-\frac 2p-6\ep)}.
\end{align}
Although the norm $||\U_1||_{\LL^\infty_T(\B^{\frac 2q}_{q,1})}$ is sufficient large for $T\leq 2^{-2n-4\ep n}$. However, we can deduce that the corresponding norm $||\U_1||_{\LL^2_T(\B^{\frac 2q}_{q,1})\cap L^1_T(\B^{\frac 2q+1}_{q,1})}$ is sufficent small when $T\leq 2^{-2n-4\ep n}$. In fact, it follows from Lemmas \ref{le3.2}-\ref{le11}, H\"{o}lder's inequality and \eqref{zhu-0}, \eqref{zhu-1} that
\bal\label{li1}\begin{split}
&\quad \ ||\U_1||_{\LL^2_T(\B^{\frac 2q}_{q,1})\cap L^1_T(\B^{\frac 2q+1}_{q,1})}\\&\leq T^{\frac12}||\U_1||_{\LL^\infty_T(\B^{\frac 2q}_{q,1})\cap \LL^2_T(\B^{\frac 2q+1}_{q,1})}\leq CT^{\frac12}||\U_0\cd \na \U_0||_{L^1_T(\B^{\frac 2q}_{q,1})}
\\&\leq CT\big(||\U_0||_{\LL^\infty_T(\B^{\frac 2p}_{p,1})}||\U_0||_{\LL^2_T(\B^{\frac 2q+1}_{q,1})}+||\U_0||_{\LL^2_T(\B^{\frac 2q}_{q,1})}||\U_0||_{\LL^\infty_T(\B^{\frac 2p+1}_{p,1})}\big)
\\&\leq CT2^{n(\frac 2q-\frac 2p+2-2\ep)}\leq C2^{n(\frac 2q-\frac 2p-6\ep)}.
\end{split}\end{align}
Similar argument as in \eqref{li1}, we have for $T\leq 2^{-2n-4\ep n}$,
\bal\label{li1-1}
||\U_1||_{\LL^2_T(\B^{\frac 2r}_{r,1})\cap L^1_T(\B^{\frac 2r+1}_{r,1})}\leq C2^{n(\frac 2r-\frac 2p-6\ep)}.
\end{align}
Now, we will show that the corresponding norm of $\U_2$ is also small when $T\leq 2^{-2n-4\ep n}$. For simplicity, we denote
\bbal
\XX_T=||\U_2||_{\LL^\infty_T(\B^{\frac 2q-1}_{q,1})\cap L^1_T(\B^{\frac 2q+1}_{q,1})}, \qquad \YY_T=||h||_{\LL^\infty_T(\B^{\frac 2q}_{q,1})}.
\end{align*}
According to Lemma \ref{le3.3}, \eqref{l-z} and \eqref{zhu-1}-\eqref{li1-1}, we have
\bal\begin{split}\label{li-1}
&\quad \ ||\U_0\cd \na \U_1+\U_1\cd \na \U_0||_{L^1_T(\B^{\frac 2q-1}_{q,1})}\\&\lesssim ||\U_0||_{\LL^2_T(\B^{\frac 2r}_{r,1})}||\U_1||_{\LL^2_T(\B^{\frac 2q}_{q,1})}
+||\U_1||_{\LL^2_T(\B^{\frac {2}{r}}_{r,1})}||\U_0||_{\LL^2_T(\B^{\frac 2q}_{q,1})}\lesssim 2^{\frac32n(\frac 2q-\frac 2p-6\ep)},
\end{split}\end{align}

\bal\begin{split}\label{li-2}
||\U_1\cd \na \U_1||_{L^1_T(\B^{\frac 2q-1}_{q,1})}&\lesssim ||\U_1||_{\LL^2_T(\B^{\frac 2q}_{q,1})}||\U_1||_{\LL^2_T(\B^{\frac 2q}_{q,1})}\lesssim 2^{2n(\frac 2q-\frac 2p-6\ep)},
\end{split}\end{align}

\bal\begin{split}\label{li-3}
& \quad \ ||\U_0\cd \na \U_2+\U_2\cd \na \U_0||_{L^1_T(\B^{\frac 2q-1}_{q,1})}\\&\lesssim
||\U_0||_{\LL^2_T(\B^{\frac 2r}_{r,1})}||\U_2||_{\LL^2_T(\B^{\frac 2q}_{q,1})}+||\U_2||_{\LL^\infty_T(\B^{\frac 2q-1}_{q,1})}||\U_0||_{L^1_T(\B^{\frac 2r+1}_{r,1})}\lesssim 2^{\frac12n(\frac 2q-\frac 2p-6\ep)}\XX_T,
\end{split}\end{align}

\bal\begin{split}\label{li-4}
||\U_1\cd \na \U_2+\U_2\cd \na \U_1||_{L^1_T(\B^{\frac 2q-1}_{q,1})}&\lesssim ||\U_1||_{\LL^2_T(\B^{\frac 2q}_{q,1})}||\U_2||_{\LL^2_T(\B^{\frac 2q}_{q,1})} \lesssim 2^{n(\frac 2q-\frac 2p-6\ep)}\XX_T,
\end{split}\end{align}
and
\bal\begin{split}\label{li-5}
||\U_2\cd \na \U_2||_{L^1_T(\B^{\frac 2q-1}_{q,1})}\lesssim||\U_2||_{\LL^2_T(\B^{\frac 2q}_{q,1})}||\U_2||_{\LL^2_T(\B^{\frac 2q}_{q,1})}\lesssim \XX^2_T.
\end{split}\end{align}
Then, combining this estimates \eqref{li-1}-\eqref{li-5}, we have
\bal\label{li-01}
||\mathbf{F}_1||_{L^1_T(\B^{\frac 2q-1}_{q,1})}\lesssim 2^{n(\frac 2q-\frac 2p-6\ep)}+2^{\frac12n(\frac 2q-\frac 2p-6\ep)}\XX_T+\XX^2_T.
\end{align}
Using Lemmas \ref{le12}-\ref{le3.3} and \eqref{zhu-2}, \eqref{li1-1}, we obtain
\bal\begin{split}\label{li-02}
&\qquad ||\na \ln(1+h)\cd\na(\U_0+\U_1+\U_2)||_{L^1_T(\B^{\frac 2q-1}_{q,1})}+||\na h||_{L^1_T(\B^{\frac 2q-1}_{q,1})}\\&\leq C||\ln(1+h)||_{L^\infty_T(\B^{\frac 2q}_{q,1})}||\U_0+\U_1+\U_2||_{L^1_T(\B^{\frac 2r+1}_{r,1})}+T||h||_{\LL^\infty_T(\B^{\frac 2q}_{q,1})}
\\&\leq C(||h||_{L^\infty_T(L^\infty)}+1)^2||h||_{L^\infty_T(\B^{\frac 2q}_{q,1})}(2^{\frac12n(\frac 2q-\frac 2p-6\ep)}+\XX_T)+2^{-2n-4\ep n}\mathrm{Y}_T
\\&\leq C(1+\YY^2_T)(2^{\frac12n(\frac 2q-\frac 2p-6\ep)}+\XX_T)\YY_T.
\end{split}\end{align}
According to Lemma \ref{le3.2} and \eqref{li-01}-\eqref{li-02}, we deduce that
\bal\label{li-x}
\XX_T\lesssim 2^{n(\frac 2q-\frac 2p-6\ep)}+2^{\frac12n(\frac 2q-\frac 2p-6\ep)}(\XX_T+\YY_T+\YY^3_T)+\XX^2_T+\YY^2_T+\YY^6_T.
\end{align}
Using the fact
\bbal
||\D \uu+h\D \uu||_{L^1_T(\B^{\frac2q}_{q,1})}&\leq C(||\uu||_{L^1_T(\B^{\frac2q+1}_{q,1})}+||h||_{\LL^\infty_T(\B^{\frac2q}_{q,1})}||\uu||_{L^1_T(\B^{\frac2q+1}_{q,1})})
\\&\leq C(2^{n(\frac2q-\frac2p-5\ep)}+\XX_T)(1+\YY_T),
\end{align*}
and applying Lemma \ref{le3.1}, we infer from \eqref{zhu-1}-\eqref{li1-1} that
\bal\begin{split}\label{li-y}
\YY_T&\leq C\exp\{C||\na \uu||_{L^1_T(\B^{\frac2q}_{q,1})}\}(||h_0||_{\B^{\frac2q}_{q,1}}+||\D \uu+h\D \uu||_{L^1_T(\B^{\frac2q}_{q,1})})
\\&\leq C\exp\{C2^{n(\frac 2q-\frac 2p-5\ep)}+C\XX_T\}(2^{n(\frac2q-\frac2p-5\ep)}+\XX_T)(1+\YY_T).
\end{split}\end{align}
Denote $T_0=2^{-2n-4\ep n}$. Then, using the continuation argument, we conclude from \eqref{li-x} and \eqref{li-y} that for any $T\leq T_0$,
\bal\label{li-con}
\XX_T\leq 2^{n(\frac 2q-\frac 2p-5\ep)}, \qquad  \YY_T\leq 2^{\frac23 n(\frac 2q-\frac 2p-5\ep)}, \qquad n\gg 1.
\end{align}
Summing up \eqref{zhu-0}, \eqref{hong-4}, \eqref{li-con}, we deduce from \eqref{l-z} that
\bbal
||\uu(T_0)||_{\B^{\frac2p-1}_{p,1}}&\geq ||\U_1(T_0)||_{\B^{\frac2p-1}_{p,1}}-||\U_0(T_0)||_{\B^{\frac2p-1}_{p,1}}-||\U_2(T_0)||_{\B^{\frac2p-1}_{p,1}}
\\&\geq c2^{n(1-\frac4p-6\ep)}-C2^{n(\frac 2q-\frac 2p-5\ep)}-C2^{-n\ep}\geq c2^{n(1-\frac4p-6\ep)}, \qquad n\gg1.
\end{align*}
This along with the fact that
\bbal
||h_0||_{\B^{\frac2p}_{p,1}}+||\uu_0||_{\B^{\frac2p-1}_{p,1}}\leq 2^{-\ep n},\qquad n\gg1,
\end{align*}
implies the result of Theorem \ref{th1}.

\vspace*{1em}
\noindent\textbf{Acknowledgements.} The authors want to thank the referees for their constructive comments and helpful suggestions which greatly improved the presentation of this paper. J. Li is supported by the National Natural Science Foundation of China (Grant No.11801090).  W. Zhu is partially supported by the National Natural Science Foundation of China (Grant No.11901092) and Natural Science Foundation of Guangdong Province (No.2017A030310634).

\end{document}